\title{Branching random walks and contact processes on Galton-Watson trees}
\author{
Wei Su\\
Department of Statistics\\
The University of Chicago\\
}
\date{\today}

\documentclass[12pt]{article}
\usepackage{amsfonts}
\usepackage{amsmath}
\usepackage{amsthm}
\usepackage{amssymb}
\usepackage{indentfirst}

\newtheorem{Theorem}{Theorem}[section]

\newtheorem{Proposition}[Theorem]{Proposition}

\begin{document}
\maketitle

\begin{abstract}
We consider branching random walks and contact processes on infinite, connected, locally finite graphs whose reproduction and infectivity rates across edges are inversely proportional to vertex degree. We show that when the ambient graph is a Galton-Watson tree then, in certain circumstances, the branching random walks and contact processes will have \textbf{weak survival} phases. We also provide bounds on critical values.
\end{abstract}

\section{Introduction}\label{sec1.introduction}

There has been considerable interest in the behavior of branching
random walks (BRW), contact processes (CP), and other related interacting
particle systems on trees and other nonamenable graphs in recent
years. These processes may exhibit a \textbf{weak survival} phase on trees and other nonamenable graphs which does not occur on the integer lattice. In the weak survival phase, the population survives
globally with positive probability, but eventually vacates any fixed
vertex with probability one. The weak survival phase of BRW has been studied, for example, in \cite{Candellero-Gilch-Muller, Hueter-Lalley,Liggett,Madras-Schinazi,Pemantle-Stacey},
and for CP in \cite{lalley,lalley-sellke:cp1,liggett2,Lyons,pemantle,Pemantle-Stacey,stacey}.

In this paper, we introduce a discrete-time BRW, where particles
reproduce as in an ordinary Galton-Watson (GW) process, regardless of their
locations in the ambient graph, and then move as in a random walk. We also introduce
a closely related version of CP. Formal definitions are given in section \ref{sec2.notation} below. 
We study BRW, which always dominates CP, in order to give natural upper bounds for CP.
Our main result (Theorem \ref{5.CP-GW}) is on the existence of a weak survival
phase for CP. 

Our BRWs and CPs differ in an important qualitative respect from those studied by
Pemantle and Stacey \cite{Pemantle-Stacey}, where the reproduction
rates depend on location (in particular, they depend linearly on the
vertex degree). This leads to rather different behaviors on
inhomogeneous graphs. For BRW, we give necessary and sufficient conditions for
the existence of the weak survival phase in terms of the spectral
radius of simple random walk (SRW) on the graph, citing results in \cite{Muller}. This requires us to calculate the
spectral radius of SRW on infinite GW trees. Then we use various techniques to provide upper and lower bounds 
for the critical values of the CP on infinite GW trees, and show that there exists a weak survival phase in certain 
circumstances.

We will deal with GW trees with offspring distribution $F_T=\{p_{k} \}_{k\geq 0}$. For conciseness and consistence, throughout this paper we will assume $p_0=0$. One thing to point out is that when $p_0>0$ most results concerning BRW in this paper can be obtained as well, however arguments for CP fail to work.

\paragraph{Outline.}The remainder of this paper is organized as
follows.  In Section~\ref{sec2.notation} we give formal definitions. General properties of BRW and its
connection to SRW are given in section~\ref{sec3.connection}. Section~\ref{sec4.CP} shows that for CP there is weak survival phase on certain GW trees.

\section{Definitions and notations}\label{sec2.notation}

All processes considered in this paper
will live on infinite, connected, locally finite
graphs. We will use  $G=(V,\mathcal{E})$ to denote such a graph,  where
$V$ is the vertex set and $\mathcal{E}$ is the edge set. 
These graphs will themselves be constructed according to some
random mechanism,  and we will use 
$G_\omega=(V_{\omega},\mathcal{E}_{\omega})$ to denote realizations of random graphs. 
In all random graph constructions we shall consider, there will be a distinguished vertex $\varrho$
designated the \emph{root}. 
Say that two vertices $x,y\in V$
are neighbors if and only if they are connected in $G$, or
equivalently $(x,y)\in \mathcal{E}$.

\textbf{Branching random walk} (BRW) is a
discrete-time stochastic process on $G$ defined in
the following way. It is a special case of discrete branching Markov chain in \cite{Muller}, with the underlying Markov chain being SRW. At time $n=0$ there is  one
particle at the root $\varrho$. Given the population at time $n$, the population at time $n+1$ is
generated in two steps (in the following definition independence means independence of other particle's behavior and the history up to time $n$):

\medskip
(1) \textbf{Particle reproduction}, where each particle currently in the
system dies and independently gives rise to a random number of offspring, according to a common distribution $F_R$.

\medskip (2) \textbf{Particle dispersal}, where each newborn particle
makes an independent SRW step from the vertex where it is born to a 
neighboring vertex on the graph. In other
words, each new particle chooses one of the neighbors of the vertex where it
is born, and then move to it. The choice is made uniformly at random.

\medskip

If the ambient graph $G$ is a tree, then it is bipartite, so at even (odd) times particles are
located only at even (odd) depths from the root.

To emphasize the dependence of the process on the underlying graph
$G$, we use $\mathbb{P}_{G}$ to denote law of BRW on $G$.
Denote the number of particles at
vertex $v$ at time $n$ by $N_n(v)$. We name the following events respectively.

\medskip (1) $\{\lim_{n\rightarrow\infty} \sum_{v\in V}N_n(v)=0\}$: extinction;

\medskip (2) $\{\liminf_{n\rightarrow\infty} \sum_{v\in V}N_n(v)\geq 1\}$: global survival;

\medskip (3) $\{\limsup_{n\rightarrow\infty}N_n(\varrho)\geq 1\}$: local survival at the vertex $\varrho$.
\medskip

Clearly the event of local survival at any vertex implies the event of global survival. As the underlying graph is connected the definition of local survival does not depend on the choice of $\varrho$. So we will use the term ``local survival" without indicating the root $\varrho$. Unless there is local survival, eventually not only every vertex is free of particles but also every finite subset.

Correspondingly, there are 3 phases.

\medskip (1) If with probability one, the BRW dies out, i.e.
\[
\mathbb{P}_{G}\left(\lim_{n\rightarrow\infty} \sum_{v\in V}N_n(v)=0\right)=1,
\]
we say the BRW is at the \textbf{subcritical} phase.

\medskip (2)
If with positive probability, the BRW survives locally (and thus globally), i.e.
\[
\mathbb{P}_{G}\left(\limsup_{n\rightarrow\infty}N_n(\varrho)\geq 1\right)>0,
\]
we say the BRW is at the \textbf{strong survival} phase. Our definition of strong survival phase corresponds to the notion of \textbf{strong recurrence} in \cite{Muller}.

\medskip (3)
If with probability one, the BRW does not survive locally; but with positive probability it survives globally, i.e.
\[
\mathbb{P}_{G}\left(\limsup_{n\rightarrow\infty}N_n(\varrho)\geq 1\right)=0,
\]
\[
\mathbb{P}_{G}\left(\liminf_{n\rightarrow\infty} \sum_{v\in
V}N_n(v)\geq 1\right)>0,
\]
we say the BRW is at the \textbf{weak survival} phase.

In a BRW (as defined above), the total number of
particles in generations $n=0,1,2,\dotsc$ evolves as a GW
process with offspring distribution $F_R=\{f_k\}_{k\geq 0}$ with mean
$\mu=\sum_k kf_k$, so global survival occurs if and only if $\mu>1$
(in the BRWs studied in \cite{Pemantle-Stacey} this is not the case).

Assume that the particle reproduction law $F_R=\{f_k\}_{k\geq0}$ is
fixed.  Then whether or not BRW on graph $G$ exhibits weak
survival phase depends only on the geometry of $G$. Our first main result (Theorem \ref{3.main})
concerns the case where $G$ is a GW tree constructed
using an offspring distribution $F_T =\{p_{k} \}_{k\geq 0}$.
It will be shown that the existence of the weak survival phase is determined by $h_{\min}$, the minimal offspring number for $F_T$, that
is, $h_{\min}=\min \{i\,:\, p_{i}>0\}$. By our assumption $h_{\min}\geq 1$.

\textbf{Continuous-time BRW} is a continuous-time Markov process
defined as follows.  At time $t=0$ there is one particle at the root
$\varrho$. Each particle gives rise to a new particle with rate
$\lambda$, meanwhile dies with rate 1, and its behavior is
independent of all other particles and the history. When a new particle is born, it
takes an instantaneous independent SRW step to one of the neighbors of the vertex
where it is born. In section \ref{sec4.CP} we will show that existence
of weak survival phase of the continuous-time BRW is essentially the
same problem as that for the discrete-time model, so it suffices to
study the discrete-time model.

\textbf{Contact process} (CP) is a continuous-time Markov process evolving in the following way (in the following definition independence means independence of other particle's behavior and the history). We start with 1 particle at $\varrho$ at time 0. Then, 

(1) Each particle gives rise to a new particle at rate $\lambda$ independently, and the newborn particle independently picks a neighboring vertex on the graph uniformly at random and makes an instantaneous movement to the picked vertex.

(2) Each particle dies with rate 1 independently.

(3) Each vertex can hold at most 1 particle. So if a newborn particle moves to a vertex where there exists a particle at that moment, the newborn vertex is removed immediately as if it was never born.

The existence of such a process is guaranteed by a modification of the classical graphical representation for CP. This CP model differs from the one defined in \cite{Pemantle-Stacey}. In homogeneous graphs (such as $\mathbb{Z}^d$ or $\mathbb{T}^d$) the two definitions of CP coincide. The only difference is that in our model we require the sum of birth rates among all directed edges going out of the same vertex be a fixed quantity $\lambda$, whereas in \cite{Pemantle-Stacey} the birth rate for each directed edge is $\lambda$, so when the underlying graph is not regular, in \cite{Pemantle-Stacey} an occupied vertex with higher degree has higher reproduction rate compared with those with lower degrees.  It is important to note that duality no longer holds in our model, because a directed edge $v_1v_2$ might have different birth rate than that of $v_2v_1$,

In particular, it is easily seen from the graphical representation that the CP is stochastically monotone in $\lambda$. We can couple contact processes simultaneously for all $\lambda>0$ on the same graph $G$.
We use $\mathbb{P}^{\lambda}_{G}$ to denote law of CP on $G$ with reproduction rate $\lambda$. Because of  monotonicity we can define
\[\lambda_g(G)=\inf \{\lambda:\mathbb{P}^{\lambda}_{G}(\forall t>0, \exists \text{ particle alive at time }t)>0\},\]
\[\lambda_{\ell}(G)=\inf \{\lambda:\mathbb{P}^{\lambda}_{G}(\forall T>0, \exists t>T,  \text{ s.t. } \exists \text{ particle at } \varrho \text{ at time }t)>0\}.\]

We say CP on $G$ has a \textbf{weak survival} phase if $\lambda_g(G)<\lambda_{\ell}(G)$.

\section{Discrete-time BRW}\label{sec3.connection}

Assume the BRW has particle reproduction law
$F_R=\{f_k\}_{k\geq0}$ with mean $\mu$. Let $(SRW_n)_{n\geq 0}$ denote the SRW started from $\varrho$, recall that the spectral
radius of SRW on a connected graph $G_{\omega}$ is given by
\[r(G)=\limsup_{n\rightarrow
\infty}\mathbb{P}_{G}(SRW_n=\varrho)^{1/n}.\] The spectral radius $r(G)$
does not depend on the choice of root $\varrho$.

In our terms, one of the main results (Theorem 3.7) of  \cite{Muller} is
\begin{Theorem}\label{2.mueller}
BRW is at the strong survival phase if and only if $\mu r(G)>1$.
\end{Theorem}

Therefore to determine whether BRW might survive locally on $G$ it suffices to compute the spectral
radius $r(G)$.

What property of graph $G$ makes its spectral radius $r(G)=1$?
One sufficient condition is the existence of arbitrarily long linear
chains -- which we will call $L$-\emph{chains} -- in the graph
$G$. An $L$-chain is defined to be a chain of vertices $\{v_{i}
\}_{0\leq i\leq L}$ such that each $v_{i}$ is a neighbor of $v_{i+1}$,
and such that all of the interior vertices $\{v_{i} \}_{1\leq i\leq
L-1}$ have degree $2$ (so their only neighbors in $G$ are
$v_{i-1}$ and $v_{i+1}$). The parameter $L$ will be called the
\emph{length} of the $L$-$chain$.

\begin{Proposition}\label{3.string}
If $G$ contains arbitrarily long $L$-chains  then
$r(G)=1$.
\end{Proposition}

\begin{proof}
This follows from proof of Lemma 3.6 in \cite{BenMue}, or Theorem 3.11 in \cite{Muller}.
\end{proof}

We can generalize the idea of $L$-chain to a finite $d$-ary ($d\geq
1$) tree of height $L$. Formally, we define a
$\mathbf{(d,L)}$\textbf{-subtree} in a graph $G$ to be a rooted
$d$-ary tree $T$ of depth $L$ embedded in $G$ in such a way that,
except for the root and the leaves (leaves are vertices at maximum
depth $L$), every vertex of $T$ has no neighbors in $G$ other
than those $d+1$ neighbors it has in the tree $T$. Observe that a
$(1,L)$-subtree is just an $L$-chain.

The relevance of $(d,L)$-subtrees to spectral radii is similar as for
$L$-chains.  Once a SRW gets into a $(d,L)$-subtree,
its depth (as viewed from the root of the $(d,L)$-subtree) behaves as a
$p$ - $q$ nearest neighbor random walk on $[0,L]$, with $p=1/(d+1)$
and $q=1-p$.

\begin{Proposition}\label{3.d-string}
If for some $d\geq 1$, $G$ contains $(d,L)$-subtrees of
arbitrary depth $L$ then
$r(G )\geq 2\sqrt{d}/(d+1)$.
\end{Proposition}
\begin{proof}
Let $Q=(Q(x,y))_{x,y\,\in V}$ be the probability transition matrix of the SRW on $G=(V,\mathcal{E})$. 
For any finite subset $F$ of $V$, denote by $Q_F$ the substochastic matrix $(Q(x,y))_{ x,y\, \in F}$ and by $r(Q_F)$ its spectral radius, then it is well known (see \cite{Benjamini-Peres} or \cite{Muller}) that if $Q$ is irreducible then $F\subset F^\prime$ implies $r(Q_F)\leq r(Q_{F^\prime})$.

Then it is easy to see that $r(G)\geq \sup_L r((d,L)\text{-subtree})=r(\mathbb{T}^d)=2\sqrt{d}/(d+1)$, where the last equality follows from Lemma 1.24 in \cite{Woess:book} and $\mathbb{T}^d$ is the regular tree with degree $d+1$.
\end{proof}

For a GW tree with offspring distribution $F_T=\{p_k\}_{k\geq 1}$ assume that  $p_d>0$
for some $d\geq 1$. It is easy to see that GW-a.e.
$G_{\omega}$ contains a $(d,L)$-subtree for every $L\in
\mathbb{N}$, because when we sequentially explore the GW tree, a vertex having a $(d,L)$-subtree attached to it in the next $L$ levels is an event with positive probability, while there are infinitely many trials and therefore eventually there will be a success.

\begin{Proposition}\label{3.gw-0}
If $p_d>0$, then GW-a.e. $G_{\omega}$
has a $(d,L)$-subtree for every $L\in\mathbb{N}$.
\end{Proposition}

Recall that the minimal offspring number $h_{\min}$ is the smallest
integer $i$ such that $p_i>0$, and by our assumption $h_{\min}\geq 1$.

\begin{Proposition}\label{3.gw}
(i) If $h_{\min}=1$, then for GW-a.e. $G_{\omega}$, $r(G_{\omega})=1$.

(ii) If $h_{\min}>1$, then for GW-a.e. $G_{\omega}$, $r(G_{\omega})=
2\sqrt{h_{\min}}/(h_{\min}+1)$.
\end{Proposition}

\begin{proof}
(i) If $h_{\min}=1$, combine Propositions \ref{3.string} and \ref{3.gw-0}.

(ii) By Propositions \ref{3.d-string} and \ref{3.gw-0}, we have
$r(G_{\omega})\geq 2\sqrt{h_{\min}}/(h_{\min}+1)$ for GW-a.e. 
$G_{\omega}$. The reverse inequality follows from exercise 11.3 in \cite{Woess:book}.
\end{proof}
Combining Theorem \ref{2.mueller} and Proposition \ref{3.gw}, we obtain
\begin{Theorem}\label{3.main}
(1) If $h_{\min}=1$, then for GW-a.e. 
$G_{\omega}$, BRW on $G_{\omega}$ has no weak survival phase for
any particle reproduction law $F_R$.

(2) If $h_{\min}>1$, then either for GW-a.e. $G_{\omega}$, BRW on $G_{\omega}$
has a weak survival phase; or for GW-a.e. $G_{\omega}$ there is no weak survival phase. More precisely, there is a weak survival phase if and only if the particle reproduction distribution $F_R=\{f_k\}_{k\geq0}$ satisfies $1<\mu=\sum_k
kf_k\leq(h_{\min}+1)/2\sqrt{h_{\min}}$.
\end{Theorem}

\textbf{Remark.} If $h_{\min}=0$, for GW-a.e. $G_{\omega}$, $r(G_{\omega})=1$ and thus there is no weak survival phase. To see this, one can use a similar argument as in the proof of Proposition \ref{3.d-string}.

\section{CP on GW trees}\label{sec4.CP}

In this section we will show that for certain augmented Galton-Watson (AGW) trees,  CP on AGW-a.e. $G_{\omega}$ exhibits weak survival phase. We will first study continuous-time BRW, then CP.

When regarding the question of (global/local) survival  of a continuous-time BRW, it is reduced to (global/local)  survival of a discrete-time BRW with geometric offspring distribution. For more details about this connection, see, for example, Section 2.2 in \cite{B-Z}. So for continuous-time BRW, its phase transition can be determined using results obtained in the last section.

Now let us focus on CP. The underlying graph we will consider are AGW trees (which means we always add an extra copy of the GW tree to the root $\varrho$). By considering AGW trees, it makes the root homogeneous with all other vertices. For example, an AGW tree with degenerated offspring distribution, $p_d=1$, is a regular tree with degree $d+1$ for each vertex; this is not true for the GW tree because the root only has $d$ neighbors. Several ergodic results are known for AGW trees, for example, \cite{Lyons-Pemantle-Peres}.  All results about BRW obtained in the last section still hold if we replace GW by AGW, because adding one copy of a GW tree to the root doesn't affect the computation of the spectral radius.

The first natural question is whether the critical values $\lambda_{\ell} (G_{\omega}), \lambda_g(G_{\omega})$ are AGW-a.s. constants? The following theorem answers this question affirmatively.

\begin{Theorem}\label{4.GW-constant}
AGW-a.s., $\lambda_{\ell}(G_{\omega})$ and $\lambda_g(G_{\omega})$ are constants.
\end{Theorem}
\begin{proof}
The proof uses the ergodic property of AGW trees. We will use CP($\lambda$) to denote the CP with infection rate $\lambda$. We explore the AGW tree from the root $\varrho$ level by level. Define $\mathcal {F}_n$ to be the $\sigma$-algebra such that $\mathcal {F}_n$ contains exactly the information of the AGW tree up to level $n$. Let $\mathcal {F}_{\infty}=\bigcup_{n\geq0} \mathcal {F}_0$.

We first show that the set
\[\mathcal{G}^\lambda=\{G_{\omega}: \text{ CP}(\lambda) \text{ survives globally with positive probability on }G_{\omega} \}\] is a measurable subset of $\mathcal {F}_{\infty}$. Let 
\[\mathcal{G}^{\lambda}_{\epsilon,N}=\{G_{\omega}: \mathbb{P}_{G_{\omega}}^{\lambda}(\text{there exists an infection trail which exits }B_{N-1}(\varrho))\geq \epsilon\}.\] This is clearly a measurable subset in $\mathcal {F}_N$.
$\mathcal{G}^{\lambda}=\bigcup_{\epsilon>0,\epsilon\in\mathbb{Q}}\bigcap_{N=1}^{\infty}\mathcal{G}^{\lambda}_{\epsilon,N}\in \mathcal {F}_{\infty}$.

Now we cite ergodic theory from \cite{Lyons-Pemantle-Peres}. In \cite{Lyons-Pemantle-Peres}, it is shown that the system (PathsInTrees, SRW$\times$AGW, $S$) (where $S$ is the shift map) is ergodic (for the definition, see \cite{Lyons-Pemantle-Peres}). It is easily seen that because global survival doesn't depend on the choice of the root $\varrho$,
$\{\text{all paths}\}\times \mathcal{G}^\lambda$ is an invariant subset of PathsInTrees under $S$. Therefore by ergodicity \[\text{SRW}\times\text{AGW} (\{\text{all paths} \}\times \mathcal{G}^\lambda)=0 \text{ or } 1,\]
which proves that under measure AGW, the set $\mathcal{G}^{\lambda}$ has measure either 0 or 1.

Similarly, we express
\[\mathcal{L}^\lambda=\{G_{\omega}: \text{ CP}(\lambda) \text{ survives locally with positive probability on }G_{\omega} \} \]
by
$\mathcal{L}^{\lambda}=\bigcup_{\epsilon>0,\epsilon\in\mathbb{Q}}\bigcap_{m=1}^{\infty}\bigcup_{N=m}^{\infty}\mathcal{L}^{\lambda}_{\epsilon,m,N}$, where
$\mathcal{L}^{\lambda}_{\epsilon,m,N}=\{G_{\omega}: \mathbb{P}_{G_{\omega}}^{\lambda}($there exists an infection trail which hits $\partial B_{m-1}(\varrho)$,
then hits $\varrho$ without exiting $B_N(\varrho)$)$\geq \epsilon\}\in \mathcal {F}_{N}$.

Then by the same argument as above, under the measure AGW, the set $\mathcal{L}^{\lambda}$ has measure either 0 or 1.
\end{proof}
Because of Theorem \ref{4.GW-constant}, from now on we will use $\lambda_{\ell}$ and $\lambda_g$ for the AGW-a.s. constants without indicating their dependences on $G_\omega$.

\begin{Theorem} \label{5.CP-GW}
If $h_{\min}\geq 4$, then the CP on $G_{\omega}$ has a weak survival phase for AGW-a.e. $G_{\omega}$.
\end{Theorem}
The proof of this theorem involves bounding $\lambda_g$ from above and bounding $\lambda_{\ell}$ from below. Proposition \ref{4.CP-lowerbound} and (i) of Proposition \ref{4.CP-finerupperbound} yield an easy proof for the case $h_{\min}\geq 6$. For the case $h_{\min}=4,5,$ we will need the more refined results stated in (ii) of Proposition \ref{4.CP-finerupperbound} and Proposition \ref{4.CP-finerlowerbound}.

Recall that we have assumed $h_{\min}\geq 1$.
\begin{Proposition}\label{4.CP-lowerbound}
$\lambda_{\ell}>(h_{\min}+1)/(2\sqrt{h_{\min}})$.
\end{Proposition}
\begin{proof}
The continuous-time BRW always dominates CP (with same $\lambda$). So if the continuous-time BRW does not survive locally, neither does CP. By \cite{B-Z}, the parameter $\lambda$ in continuous-time BRW serves as $\mu$ in the corresponding discrete-time BRW. From Theorem \ref{2.mueller} and Proposition \ref{3.gw} (and an easy argument that by switching to AGW tree the spectral radius is unchanged), $\lambda_{\ell}> 1/r(G_{\omega})=(h_{\min}+1)/(2\sqrt{h_{\min}})$ for AGW-a.e. $G_{\omega}$.
\end{proof}
Next we give an upper bound for $\lambda_g$ for AGW tree.
\begin{Proposition}\label{4.CP-finerupperbound}
Suppose $X$ is distributed as $F_T$. If $\lambda$ satisfies the following inequality ($\mathbb{E}_X$ means taking expectation w.r.t. $X$)
\[\mathbb{E}_X \left(\lambda X(\lambda+X+1)^{-1}\left(1-\frac{\lambda}{\lambda+X+1}\frac{1}{2+\lambda/(h_{\min}+1)}\right)^{-1}\right)>1,\]
then $\lambda_g\leq\lambda$.

Furthermore,

(i) if $h_{\min}\geq 2$, then $\lambda_g\leq (h_{\min}+1)/(d_{\min-1})$;

(ii) in particular, if $h_{\min}=4,$ then $\lambda_g\leq 1.46$; 
if $h_{\min}=5,$ then $\lambda_g\leq 1.35$.
\end{Proposition}
\begin{proof}
The strategy is to construct a supercritical GW process which is dominated by CP. We will build a ``block" in the AGW tree , run the CP within this block, retain the particles at the bottom of the block and use each of them as ``seed" for the CP on the next block.

The root $\varrho$ has $1+X$ neighbors, among them 1 parent and $X$ children, where $X$ is distributed as $F_T$. Imagine the parent of $\varrho$ to be at level -1, $\varrho$ at level 0, and the $X$ children at level 1. For any descendant of $\varrho$, its level is defined to be its graph distance to $\varrho$.

We build the GW process $(|\xi_m|)_{m\geq 0}$ as follows, where $\xi_m$ is set-valued.  Fix a positive integer $n\geq 1$ and let $\xi_0=\{\varrho\}$ (and $|\xi_0|=1$). 

\textbf{Stage 1}: explore the next $n+1$ levels of the AGW tree, regard them as a block.

\textbf{Stage 2}: run CP on this $(n+1)$-level block. This means we do not allow $\rho$ to infect its parent. Keep in mind that the only initially infected vertex is $\varrho$. Those vertices at the bottom (the $(n+1)$-st level) that ever get infected are regarded as $\xi_1$. We ``freeze"  particles at the bottom level until all the other particles die out.

When all particles die out on this $(n+1)$-level block except for those ``frozen" ones at the bottom level, we repeat stage 1 and 2 using these infected vertices as roots. This gives a GW process $(|\xi_m|)_m$ which is dominated by the original CP (which means if $(|\xi_m|)_m$ survives, so does the original CP), because the infection trails in $(\xi_m)_m$ are completely contained in the original CP. Suppose we are able to show that $\mathbb{E}|\xi_1|$ is greater than $1$ for some $\lambda$, then it implies the CP survives globally with positive probability for this $\lambda$ and from Theorem \ref{4.GW-constant}, $\lambda_g\leq \lambda$.

Now consider a vertex $v_{n+1}$ at the $(n+1)$-th level of a block. Suppose the geodesic connecting $v_{n+1}$ and $\varrho=v_0$ is $v_0,v_1,\dots,v_{n},v_{n+1}$, and suppose $v_i$ has $X_i$ offsprings in $G_{\omega}$. At time 0 only $v_0$ is infected. Consider the following events.

(1) $v_i$ infects $v_{i+1}$, and then the particle at $v_{i+1}$ dies before either the particle at $v_i$ dies or $v_{i+1}$ infects $v_{i+2}$; call this event $A_i, 0\leq i\leq n-1$.

(2) $v_i$ infects $v_{i+1}$; call this event $B_i, 0\leq i \leq n$.

In order that $v_{n+1}$ gets infected, we could have the following events happen in order: $A_0$ happens $m_0$ times, and then $B_0$ happens once;  then $A_1$ happens $m_1$ times, and then $B_1$ happens once; ...; $A_{n-1}$ happens $m_{n-1}$ times, and then $B_{n-1}$ happens once; finally $B_n$ happens once and $v_{n+1}$ now gets infected. Denote the above sequences of events by an $n$-tuple $(m_0, m_1, \dots, m_{n-1} )$ where each component is a nonnegative integer. It is easy to see that different $n$-tuples correspond to disjoint events. Now let us compute the probability of observing a specific $n$-tuple $(m_0, m_1, \dots, m_{n-1})$. This means we first observe event $A_0$ happens $m_0$ times. The probability that $A_0$ happens is $q_0=\frac{\lambda/(X_0+1)}{\lambda/(X_0+1)+1}\times\frac{1}{1+1+\lambda/(X_1+1)}$. The first factor is because we need $v_0$ infects $v_1$ before the particle at $v_0$ dies; this means for 2 independent Poisson processes with rates $\lambda/(X_0+1)$ and 1, the one with rate $\lambda/(X_0+1)$ has to give the first occurrence before the other. The second factor is because we need the particle at $v_1$ dies before the particle at $v_0$ dies or $v_1$ infects $v_2$; this means a Poisson process with rate 1 has to give the first occurrence before the other 2 independent processes with rates 1 and $\lambda/(X_1+1)$. The probability of $B_0$ happens is $r_0=\frac{\lambda/(X_0+1)}{\lambda/(X_0+1)+1}$ which is already explained. Therefore the probability of observing the tuple $(m_0, m_1, \dots, m_{n-1} )$ is $q_0^{m_0}q_1^{m_1}\dots q_{n-1}^{m_{n-1}}r_0r_1\dots r_{n}$, where $q_i=\frac{\lambda/(X_i+1)}{\lambda/(X_i+1)+1}\times\frac{1}{1+1+\lambda/(X_{i+1}+1)}$, $r_i=\frac{\lambda/(X_i+1)}{\lambda/(X_i+1)+1}$.

So the probability that $v_{n+1}$ eventually gets infected, is at least
\[\begin{aligned}&\sum_{m_0\in\mathbb{N}}\dots\sum_{m_{n-1}\in\mathbb{N}}q_0^{m_0}q_1^{m_1}\dots q_{n-1}^{m_{n-1}}r_0r_1\dots r_{n}\\
=& \frac{1}{1-q_0}\frac{1}{1-q_1}\dots\frac{1}{1-q_{n-1}} r_0r_1\dots r_{n}.\\\end{aligned}\]
But $v_{n}$ has $X_n$ children at the $(n+1)$-st level, so the expected number (given $(X_i)_{0\leq i\leq n}$) of infected children of $v_{n}$ is at least
\[X_n\frac{1}{1-q_0}\frac{1}{1-q_1}\dots\frac{1}{1-q_{n-1}} r_0r_1\dots r_{n}.\]
If we keep counting infected descendants at the $(n+1)$-st level of $v_{n-1},v_{n-2},$ $\dots,v_0$, a simple induction argument shows that the expected  total number of infected vertices at the (n+1)-st level is given by
\begin{equation}\label{4.expectation}
\mathbb{E}_{X_0,X_1,\dots,X_n}\left(X_0X_1\dots X_n\frac{1}{1-q_0}\frac{1}{1-q_1}\dots\frac{1}{1-q_{n-1}} r_0r_1\dots r_{n}\right),\end{equation}
where $X_0,X_1,\dots,X_n$ are i.i.d. with distribution $F_T$. Now we bound (\ref{4.expectation}) from below.
Notice that since $X_{i+1}\geq h_{\min}$,
\[\begin{aligned}
&\frac{1}{1-q_i}=\left(1-\frac{\lambda}{\lambda+X_i+1}\times\frac{1}{2+\lambda/(X_{i+1}+1)}\right)^{-1}\\
& \geq \left(1-\frac{\lambda}{\lambda+X_i+1}\times\frac{1}{2+\lambda/(h_{\min}+1)}\right)^{-1}.
\end{aligned}\]
Define
\[f_{h_{\min}}(x,\lambda)=\lambda x(\lambda+x+1)^{-1}\left(1-\frac{\lambda}{\lambda+x+1}\frac{1}{2+\lambda/(h_{\min}+1)}\right)^{-1}.\]
So (\ref{4.expectation}) is at least
\begin{equation}\label{4.expectation2}
\begin{aligned}
&\mathbb{E}_{X_0,X_1,\dots,X_n} \left( \frac{X_n\lambda}{\lambda+X_n+1} \prod_{i=0}^{n-1} f_{h_{\min}}\left(X_i,\lambda\right) \right)\\
=& \left(\mathbb{E}_X f_{h_{\min}}(X,\lambda) \right)^n  \times \mathbb{E}_X \left(\frac{\lambda X}{\lambda+X+1}\right):=\, I^n \times II.\end{aligned}
\end{equation}

Therefore as long as $I>1$, we can choose $n$ large enough so that (\ref{4.expectation2}) is great than 1.

Now we will show (i) and (ii).

(i): Notice that
\begin{equation}\label{4.equation3}
\begin{aligned}
&\mathbb{E}_X  f_{h_{\min}}(X,\lambda)\geq \mathbb{E}_X \left(\frac{\lambda X}{\lambda+X+1}\right)\geq \frac{\lambda h_{\min}}{\lambda+h_{\min}+1}, 
\end{aligned}\end{equation}
because the function $\lambda t/(\lambda+t+1)$ is increasing in $t$ when $t>0$. So plug  $\lambda=(h_{\min}+1)/(h_{\min}-1)$ into the rightmost expression in (\ref{4.equation3}) and we can verify that $(h_{\min}+1)/(h_{\min}-1)$ is an upper bound for $\lambda_g$.

(ii):
For the case $h_{\min}=4,5$, it is easy to verify that $f_{h_{\min}}(x,\lambda)$ is increasing in $x$. Therefore if we pick $\lambda$ such that $f_{h_{\min}}(h_{\min},\lambda)>1$, then we get the desired inequality
\[\mathbb{E}_X f_{h_{\min}}(X,\lambda)\geq \mathbb{E}_X f_{h_{\min}}(h_{\min},\lambda)>1.\]

So now we need to find $\lambda$ as small as possible such that $f_{h_{\min}}(h_{\min},\lambda)>1$. It can be verified that when $h_{\min}=4$ then $\lambda$ can be chosen to be 1.46; when $h_{\min}=5$ then $\lambda$ can be chosen to be 1.35.
\end{proof}
\textbf{Remark.} Even if $h_{\min}< 4$, if $F_T$ has heavy tail such that $\mathbb{E}_X f_{h_{\min}}(X,\lambda)$$>1$ then from Proposition~\ref{4.CP-finerupperbound} we still have $\lambda>\lambda_g$.

Next we give a tighter lower bound of $\lambda_{\ell}$. The method we use in Proposition~\ref{4.CP-finerlowerbound} can be used to improve the lower bound in Proposition~\ref{4.CP-lowerbound}. However for the purpose of separating $\lambda_g$ and $\lambda_{\ell}$, Proposition \ref{4.CP-lowerbound} is enough when $h_{\min}\geq6$, so we only state the result in the case $h_{\min}=4,5$.
\begin{Proposition}\label{4.CP-finerlowerbound}
If $h_{\min}=4$, then $\lambda_{\ell}\geq 1.50$.
If $h_{\min}=5$, then $\lambda_{\ell}\geq 1.59$.
\end{Proposition}

\begin{proof}
We modify the proof of Theorem 2.2 in \cite{pemantle}. Denote the infected vertices set at time $t$ by $\xi(t)$, which is a subset of $V_{\omega}$. The idea is to construct a positive weight function $W(v)$, such that
\[W(\xi(t))=\sum_{v\in V_{\omega}} W(v)\mathbf{1}_{\{v\in \xi(t)\}}\] is a nonnegative supermartingale whose expectation decays exponentially in $t$. Then it is easy to see that local survival cannot happen. This is because when $\varrho$ is infected, $W(\xi(t))$ is at least $W(\varrho)$, we can apply Markov inequality together with the fact that $\mathbb{E}W(\xi(t))$ decays exponentially to conclude that the chance of $\varrho\in \xi(t)$ decays exponentially as $t$ approaches infinity.

Now for a vertex whose distance from the root $\varrho$ is $k$ and who has $n_v$ children in $G_{\omega}$ (and 1 parent), define
\[W(v)=r^k(1-b\theta_1(v)),\]
where $\theta_1(v)=\mathbf{1}_{\{\text{parent of }v \in \xi(t)\}}$, and $0<r<1,0<b<1$ are constants to be determined. Notice that $W(\xi(t))$ is $\xi(t)$-measurable. Let $\theta_2(v)=\#\{\text{children of }v \in \xi(t)\}$. Let's calculate the contribution of any changes (infection/recovery) caused by $v$ to the total weight $W(\xi(t))$ in time interval $(t,t+dt)$.

\textbf{Case 1}: with rate 1, the particle at $v$ dies. This causes a loss of \[r^k\left(1-b\theta_1(v)\right)\] at $v$, but a gain of \[\theta_2(v)r^{k+1}b\] by the increased weights of the infected children of $v$.

\textbf{Case 2}: with rate $\frac{1-\theta_1(v)}{n_v+1}\lambda$, $v$ infects its parent. The parent will gain at most (depending whether the grandparent of $v$ is infected)
\[r^{k-1},\]
while $v$ loses $br^k$.

\textbf{Case 3}: with rate $\frac{n_v-\theta_2(v)}{n_v+1}\lambda$, $v$ infects its (uninfected) children. This causes a gain of at most \[(1-b)r^{k+1}\] from $v$'s child, while possibly causing some loss due to $v$'s grandchildren.

Combine all 3 possible cases, from $t$ to $t+dt$, the expected change of total weight due to changes related to $v$ has an upper bound
\[
\begin{aligned}
&dt\cdot r^k\left(-1+b\theta_1(v)+br\theta_2(v)+\frac{1-\theta_1(v)}{n_v+1}\lambda(\frac{1}{r}-b)+\frac{n_v-\theta_2(v)}{n_v+1}\lambda r(1-b)\right)\\
&:=dt\cdot r^ku(v).
\end{aligned}\]
Suppose we were able to show that $u(v)<-\epsilon$ for all values of $n_v, \theta_1(v),\theta_2(v)$ for some positive $\epsilon$, then summing over $\xi(t)$, we would be able to show $\mathbb{E} (W(\xi(t+dt))|\,\xi(t))\leq W(\xi(t))-dt\cdot\epsilon(1-b)W(\xi(t))$ and thus the exponential decay of $\mathbb{E} W(\xi(t))$. However this is not possible. An alternative solution is given as follows. Define
\begin{equation}\label{4.weight}
\begin{aligned}
&U(v)=u(v)+\frac{\theta_1(v)c}{r}-\theta_2(v)c,
\end{aligned}
\end{equation}
where $c$ is another constant to be determined. The sum over $\xi(t)$ of $U(v)r^k$ is the same as the sum of $u(v)r^k$, because the two additional terms will be canceled in each infected parent-child pair in the sum.

Now we will choose proper constants $\lambda,r,b,c$ such that $U(v)<-\epsilon$ for some positive $\epsilon$. Notice that by the definition of $h_{\min}$, we always have $n_v\geq h_{\min}$. Also notice that (\ref{4.weight}) is linear in $\theta_1(v),\theta_2(v)$, where $\theta_1(v)$ ranges in $\{0,1\}$, $\theta_2(v)$ ranges in $\{0,1,\dots,n_v\}$. Because linear functions always take extreme values at boundaries, it suffices to consider the following 4 extreme combinations: $(\theta_1(v),\theta_2(v))=(0,0),(0,n_v),(1,0),(1,n_v)$. Requiring $1+U(v)<1-\epsilon$ is equivalent to

\begin{equation} \label{4.inequality2}
\left\{ \begin{aligned}
\frac{\lambda}{n_v+1}\left(\frac{1}{r}-b\right)+\frac{\lambda}{n_v+1}n_vr(1-b) &<1-\epsilon ,\\
(br-c)n_v+\frac{\lambda}{n_v+1}\left(\frac{1}{r}-b\right) &<1-\epsilon ,\\
b+\frac{\lambda}{n_v+1}n_vr(1-b)+\frac{c}{r}&<1-\epsilon ,\\
b+(br-c)n_v+\frac{c}{r}&<1-\epsilon .\\
\end{aligned} \right.
\end{equation} 

We need (\ref{4.inequality2}) to hold for all $n_v\geq h_{\min}$.  As long as we require $br-c\leq0, b<1$, the second and the fourth inequalities are redundant. Furthermore if we let $\nu=\lambda/(h_{\min}+1)$, we obtain 
\begin{equation} \label{4.inequality3}
\left\{ \begin{aligned}
(h_{\min}+1)\nu\left(\frac{1}{n_v+1}\left(\frac{1}{r}-b\right)+\frac{n_v}{n_v+1}r(1-b)\right) &<1-\epsilon ,\\
b+\frac{c}{r}+\frac{n_v}{n_v+1}(h_{\min}+1)\nu r(1-b)&<1-\epsilon .\\
\end{aligned} \right.
\end{equation}

We need (\ref{4.inequality3}) to hold for all $n_v\geq h_{\min}$. Since $1/r-b>r(1-b)$ for $b,r<1$ the LHS of the first inequality in (\ref{4.inequality3}) is maximized (as a function of $n_v$) when $n_v=h_{\min}$ (because now it puts the largest possible weight on $1/r-b$). The LHS of the second inequality in (\ref{4.inequality3}) is obviously bounded from above by
\[b+\frac{c}{r}+(h_{\min}+1)\nu r(1-b).\]
Therefore to show that (\ref{4.inequality3}) holds for every $n_v\geq h_{\min}$ (possibly infinitely many inequalities), now it suffices to show the following two inequalities
\begin{equation} \label{4.inequality4}
\left\{ \begin{aligned}
\nu\left(\frac{1}{r}-b+h_{\min}r(1-b)\right) &<1-\epsilon ,\\
b+\frac{c}{r}+(h_{\min}+1)\nu r(1-b)&<1-\epsilon ,\\
\end{aligned} \right.
\end{equation}
for some proper choice of $\nu,b,r,c$ with constraints $br\leq c,b,r<1$.

It can be verified that:
\begin{itemize}
\item when $h_{\min}=4$, the choice of
$\nu=0.3,\,r=0.437,\,b=0.256,\,c=br,\,\epsilon=0.0001\,,\lambda=\nu(h_{\min}+1)=1.5$
satisfies (\ref{4.inequality4}), which implies when $h_{\min}=4$, $\lambda_{\ell}\geq1.5$;
\item when $h_{\min}=5$, the choice of
$\nu=0.265,\,r=0.397,\,b=0.264,\,c=br,\,\epsilon=0.0001\,,\lambda=\nu(h_{\min}+1)=1.59$
satisfies (\ref{4.inequality4}), which implies when $h_{\min}=5$, $\lambda_{\ell}\geq1.59$.
\end{itemize}
\end{proof}

Unfortunately this method doesn't give tight enough lower bounds of $\lambda_{\ell}$ in the case $h_{\min}\leq 3$ to show the existence of weak survival phase.

\section*{Acknowledgements}
The author would like to thank his advisor Professor Steven Lalley for suggesting this problem and many useful discussions.

\bibliographystyle{plain}
\bibliography{mainbib}

\end{document}